 \documentclass[11pt,twoside]{amsart}
\usepackage{amsmath,amsfonts,amscd,amsthm}
\usepackage{amssymb}
\usepackage{color}
\newtheorem{thm}{Theorem}[section]
\newtheorem{prop}[thm]{Proposition}
\newtheorem{lem}[thm]{Lemma}
\newtheorem{cor}[thm]{Corollary}

\theoremstyle{definition}
\newtheorem{defin}[thm]{Definition}

\theoremstyle{remark}
\newtheorem{rem}[thm]{Remark}

\numberwithin{equation}{section}

 \providecommand\ufootnote[1]{{\let\thefootnote\relax\footnote[0]{#1}}}

\newcommand{\ac}{\mathcal A}

\newcommand{\cc}{\mathcal C}
\newcommand{\dc}{\mathcal D}
\newcommand{\ec}{\mathcal E}

\newcommand{\oc}{\mathcal O}

\newcommand{\nb}{\mathbb N}

\newcommand{\cb}{\mathbb C}

\newcommand{\pb}{\mathbb P}

\newcommand{\ci}{{\mathcal C}^\infty}

\newcommand{\ol}{\overline}

\newcommand{\opa}{\ol\partial}
\newcommand{\wt}{\widetilde}

\begin{document}

\title{Holomorphic approximation via Dolbeault cohomology}


\thanks{The first author would like to thank the university of Notre Dame for its support during her stay in April 2019.  The second author was partially supported by  National Science Foundation grant  DMS-1700003}

\author{Christine Laurent-Thi\'ebaut}
\address{Universit\'e Grenoble-Alpes, Institut Fourier, CS 40700
38058 Grenoble cedex 9, France and CNRS UMR 5582, Institut Fourier,
Saint-Martin d'H\`eres, F-38402, France}
\email{christine.laurent@univ-grenoble-alpes.fr}
\author{Mei-Chi Shaw}
\address{Department of Mathematics,  University of Notre Dame, Notre Dame, IN 46656, USA. }
\email{mei-chi.shaw.1@nd.edu }

\maketitle


\bibliographystyle{amsplain}
\begin{abstract}
 The purpose of this paper is to study holomorphic approximation and approximation of $\opa$-closed forms in complex manifolds of complex dimension $n\ge 1$.  We consider extensions of the classical Runge theorem and   the Mergelyan property to domains in complex manifolds for the $\ci$-smooth and the $L^2$ topology. We   characterize the Runge or Mergelyan  property   in terms of certain   Dolbeault cohomology groups and some geometric sufficient conditions are given.
\end{abstract}
\medskip\medskip

Holomorphic approximation is a fundamental subject in complex analysis.
The Runge theorem asserts that, if $K$ is a compact subset of an open Riemann surface $X$ such that $X\setminus K$ has no relatively compact connected components, then every holomorphic function on a neighborhood of $K$ can be approximated uniformly on $K$ by holomorphic functions on $X$.

If $K$ is a compact subset of an open Riemann surface $X$, we denote by $\ac(K)$ the space of continuous functions on $K$, which are holomorphic in the interior of $K$. Then the Mergelyan theorem asserts that, if $K$ is such that $X\setminus K$ has no relatively compact connected components, then every function in $\ac(K)$ can be approximated uniformly on $K$ by holomorphic functions on $X$.

Holomorphic approximation in one complex variable has been    studied and well understood, while the analogous problems in several variables are much less understood with many open questions.     An   up-to-date account of  the history and recent development  of holomorphic approximation in one and several variables can be found in the   paper by J.E. Fornaess, F. Forstneric and E.F. Wold \cite{FoForWo}.

In this paper we will consider holomorphic approximation in complex manifolds of higher complex dimension and also approximation of $\opa$-closed forms for different topologies like the uniform or the smooth topology on compact subsets or the $L^2$ topology.
The aim  is to characterize different types of holomorphic or $\opa$-closed approximation in a subdomain of a complex manifold using properties of the Dolbeault cohomology   with compact or prescribed support in the domain  or using properties of the Dolbeault cohomology of the complement of the domain with respect to some family of support.

If $M$ is a complex manifold, we denote by $H^{p,q}_c(M)$ the Dolbeault cohomology group with compact support of bidegree $(p,q)$ in $M$. Let $D\subset\subset X$ be relatively compact domain in a complex manifold $X$, for any neighborhood $V$ of $X\setminus D$ the family $\Phi$ of supports in $V$ consists of all closed subsets $F$ of $V$ such that $F\cup\ol D$ is a compact subset of $X$. For $0\leq p,q\leq n$,
$H^{p,q}_\Phi(X\setminus D)=0$ means that for any neighborhood $V$ of $X\setminus D$ and for any $\opa$-closed $(p,q)$-form $f\in\ci_{p,q}(V)$ with ${\rm supp}~f\in\Phi$, there exist a neighborhood $U\subset V$ of $X\setminus D$ and a $(p,q-1)$-form $g\in\ci_{p,q-1}(U)$ with ${\rm supp}~g\subset F\in\Phi$ such that $\opa g=f$ on $U$.

In the first part, in the spirit of the Runge theorem, we get the following result:

\begin{thm}
Let $X$ be a Stein manifold of complex dimension $n\geq 2$ and $D\subset\subset X$ a relatively compact domain in $X$ with connected complement, then the following assertions are equivalent

(i) $D$ is pseudoconvex and any holomorphic function on $D$ can be approximate by holomorphic functions in $X$ uniformly on  compact subsets of $D$;

(ii) $H^{n,r}_c(D)=0$, for $2\leq r\leq n-1$, and the natural map $H^{n,n}_c(D)\to H^{n,n}_c(X)$ is injective;

(iii) $H^{n,q}_\Phi(X\setminus D)=0$ for all $1\leq q\leq n-1$.
\end{thm}
{\parindent=0pt More generally,  we obtain a sufficient geometric condition for the approximation of $\opa$-closed forms.}

\begin{thm}
Let $X$ be a non-compact complex manifold of complex dimension $n\geq 2$, $D\subset\subset X$ a relatively compact domain in $X$ and $q$ a fixed integer such that $0\leq q\leq n-2$. Assume  that, for any neighborhood $V$ of $X\setminus D$, there exists a domain $\Omega$ such that $X\setminus V\subset\Omega\subset D$ and $X$ is a $(q+1)$-convex extension of $\Omega$. Then, for any $0\leq p\leq n$, the space $Z^{p,q}_\infty(X)$ of $\opa$-closed smooth $(p,q)$-forms on $X$ is dense in the space $Z^{p,q}_\infty(D)$ of $\opa$-closed $(p,q)$-forms on $D$ for the topology of uniform convergence of the form and all its derivatives on compact subsets of $D$.
\end{thm}

We also  give an alternative proof of the Oka-Weil theorem.

\begin{thm}
Let $X$ be a Stein manifold of complex dimension $n\geq 2$ and  $K$ a compact subset of $X$. Assume $K$ is $\oc(X)$-convex, then every holomorphic function on a neighborhood of $K$ can be approximated uniformly on $K$ by holomorphic functions on $X$.
\end{thm}

In the second part we use the  solution of the $\opa$ equation with prescribed support and the associated Serre duality to study the holomorphic approximation of holomorphic functions on a relatively compact domain $D$ of a complex manifold $X$, which are smooth up to the boundary or in  $L^2(D)$, by holomorphic functions in $X$ (in the spirit of the Mergelyan theorem) or in a neighborhood of $\ol D$ (Mergelyan property) for the associated topology.

In particular, in the $L^2$ setting, we prove the following characterization :

\begin{thm}
Let $X$ be a Stein manifold of complex dimension $n\geq 2$  and $D\subset\subset X$ a relatively compact domain in $X$ with Lipschitz boundary such that $X\setminus D$ is connected. Then
the following assertions are equivalent:

(i)  $D$ is pseudoconvex and $L^2$ holomorphic functions in $D$ can be approximated by holomorphic functions in $X$ for the $L^2$ topology on $D$;

(ii) $H^{n,r}_{\ol D, L^2}(X)=0$, for $2\leq r\leq n-1$, and the natural map
$H^{n,n}_{\ol D,L^2}(X)\to H^{n,n}_c(X)$ is injective;

(iii) ${H}^{n,q}_{\Phi,W^1}(X\setminus D)=0$, for all $1\leq q\leq n-1$.
\end{thm}

We also get the following $L^2$ version of the Oka-Weil theorem:

\begin{thm}\label{geomL2}
Let $X$ be a Stein manifold of complex dimension $n\geq 2$ and let $ D\subset\subset X$ be a
relatively compact pseudoconvex domain with  Lipschitz  boundary in $X$. Assume the closure $\overline D$ of $D$ has a $\oc(X)$-convex neighborhood basis, then $L^2$ holomorphic functions in $D$ can be approximated by holomorphic functions in $X$ for the $L^2$ topology on $D$.
\end{thm}

We will  give an example of a strictly pseudoconvex domain $\Omega$ with smooth boundary in $\cb^2$, whose closure fails to be a Runge compact subset in $\cb^2$. In fact, $\Omega$ is also not Runge.

\medskip
\noindent{\bf Example:}
Consider the domain
$$\Omega=\{(z_1,z_2)\in\cb^2~|~|z_1|^2+(|z_2|^2-2)^2<\frac{1}{4}\}.$$
It is easy to see that the  domain $\Omega$ is a bounded strictly pseudoconvex domain with smooth boundary in $\cb^2$. But $\ol\Omega$ is not $\oc(\cb^2)$-convex. The $\oc(\cb^2)$-convex hull of $\ol\Omega$ is the union of $\ol\Omega$ and the bidisc $\ol\Delta(0,\frac{1}{2})\times\ol\Delta(0,\sqrt{2})$.

We first show that its closure $\ol\Omega$ is not Runge compact subset in $\cb^2$. To see this consider the function
$$g=\frac 1{z_2}.$$
Then $g$  is holomorphic in a neighborhood of $\ol\Omega$, but
$g$ cannot be approximated by functions in $\oc(\cb^2)$ uniformly on  $\ol\Omega$. Note that the domain $\Omega$ fails  to be Runge in $\cb^2$ also.

Thus strict pseudoconvexity is not enough to guarantee the Runge property.   This  example illustrates the subtle nature of holomorphic approximation in several variables.  In comparison, holomorphic approximation in one complex variable  is much easier to describe. Our results can also be applied to the one complex variable case, which we summarize  at the end of the paper.

\section{About Runge domains in complex manifolds}\label{srunge}

Let $X$ be a complex manifold of complex dimension $n\geq 1$ and $D\subset\subset X$ a relatively compact domain in $X$.
Recall that the domain $D$ is Runge in $X$ if and only if the space $\oc(X)$ of holomorphic functions on $X$ is dense in the space $\oc(D)$ of holomorphic functions on $D$ for the topology of uniform convergence on compact subsets of $D$. We will extend the Runge property to $\opa$-closed $(p,q)$-forms.

\begin{defin}
Let $q$ be a fixed integer such that $0\leq q\leq n$. A relatively compact domain $D$ in $X$ is called a \emph{$q$-Runge domain} in $X$ if and only if, for any $0\leq p\leq n$, the space $Z^{p,q}_\infty(X)$ of $\opa$-closed smooth $(p,q)$-forms on $X$ is dense in the space $Z^{p,q}_\infty(D)$ of $\opa$-closed $(p,q)$-forms on $D$ for the topology of uniform convergence of the form and all its derivatives on compact subsets of $D$.
\end{defin}

Note that, for any $0\leq p\leq n$, $Z^{p,n}_\infty(D)=\ci_{p,n}(D)$, so any domain $D\subset\subset X$ is $n$-Runge. If $q=0$, for any $0\leq p\leq n$, $Z^{p,q}_\infty(D)$ is the space of holomorphic $p$-forms and in this case the smooth topology coincides with the uniform convergence on compact subsets, so $0$-Runge domains coincide with classical Runge domains.

\subsection{Characterization of Runge domains using Dolbeault cohomology groups}

\begin{thm}\label{runge}
Let $X$ be a non-compact complex manifold of complex dimension $n\geq 1$, $D\subset\subset X$ a relatively compact domain in $X$ and $q$ a fixed integer such that $0\leq q\leq n-1$. Assume  that, for any $0\leq p\leq n$, $H^{n-p,n-q}_c(X)$  and $H^{n-p,n-q}_c(D)$ are Hausdorff. Then $D$ is a $q$-Runge domain in $X$ if and only if, for any $0\leq p\leq n$, the natural map
$$H^{n-p,n-q}_c(D)\to H^{n-p,n-q}_c(X).$$
is injective.
\end{thm}
\begin{proof}
Assume $D$ is $q$-Runge in $X$ and let $f\in Z^{n-p,n-q}_\infty(D)$ with compact support in $D$. We assume that  the cohomological class $[f]$ of $f$ vanishes in $H^{n-p,n-q}_c(X)$, which means that there exists $g\in\dc^{n-p,n-q-1}(X)$ such that $f=\opa g$. Since $H^{n-p,n-q}_c(D)$ is Hausdorff, then $[f]=0$ in $H^{n-p,n-q}_c(D)$ if and only if, for any $\opa$-closed $(p,q)$-form $\varphi\in Z^{p,q}_\infty(D)$, we have $\int_D \varphi\wedge f=0$. But, as $D$ is $q$-Runge in $X$, there exists a sequence $(\varphi_k)_{k\in\nb}$ of $\opa$-closed $(p,q)$-form in $X$ which converges to $\varphi$ uniformly on compact subsets of $D$, in particular on the support of $f$. So
$$\int_D \varphi\wedge f=\lim_{k\to\infty}\int_X \varphi_k\wedge f=\lim_{k\to\infty}\int_X \varphi_k\wedge \opa g=\pm\lim_{k\to\infty}\int_X \opa\varphi_k\wedge g=0.$$

Conversely, by the Hahn-Banach theorem, it is sufficient to prove that, for any $\opa$-closed $(p,q)$-form $g\in Z^{p,q}_\infty(D)$ and any $(n-p,n-q)$-current $T$ with compact support in $D$ such that $<T,f>=0$ for any $\opa$-closed $(p,q)$-form $f\in Z^{p,q}_\infty(X)$, we have $<T,g>=0$. Since $H^{n-p,n-q}_c(X)$ is Hausdorff, the hypothesis on $T$ implies that there exists an $(n-p,n-q-1)$-current $S$ with compact support in $X$ such that $T=\opa S$. The injectivity of the natural map
$H^{n-p,n-q}_c(D)\to H^{n-p,n-q}_c(X)$ implies that there exists an $(n-p,n-q-1)$-current $U$ with compact support in $D$ such that $T=\opa U$. Hence, for any $g\in Z^{p,q}_\infty(D)$, we get
$$<T,g>=<\opa U,g>=\pm<U,\opa g>=0.$$
\end{proof}

\begin{cor}\label{0runge}
Let $X$ be a non-compact complex manifold of complex dimension $n\geq 1$ and $D\subset\subset X$ a relatively compact domain in $X$ such that both $H^{n,n}_c(X)$  and $H^{n,n}_c(D)$ are Hausdorff. Then $D$ is a Runge domain in $X$ if and only if the natural map
$$H^{n,n}_c(D)\to H^{n,n}_c(X).$$
is injective.
\end{cor}

It follows from the Serre duality that $H^{n-p,n-q}_c(X)$  and $H^{n-p,n-q}_c(D)$ are Hausdorff if and only if $H^{p,q+1}(X)$  and $H^{p,q+1}(D)$ are Hausdorff.
The latter condition holds in particular if these groups are finite dimensional. From the Andreotti-Grauert theory (see e.g. \cite{HeLe2}), for a complex manifold $M$, we get that $H^{p,q+1}(M)$ is finite dimensional if $M$ is either an $r$-convex complex manifold, $1\leq r\leq q+1$, or an $r$-concave complex manifolds, $1\leq r\leq n-q-2$. Moreover, by the Andreotti-Vesentini theorem (see e.g. \cite{HeLe2}, section 19), $H^{p,q+1}(M)$ is Hausdorff if $M$ is $r$-concave with $r=n-q-1$.
Finally for a $r$-convex-$s$-concave complex manifold $M$, if $r-1\leq q\leq n-s-1$, then $H^{p,q+1}(M)$ is Hausdorff, for any $0\leq p\leq n$.

\begin{cor}\label{rungeStein}
Let $X$ be a complex manifold of complex dimension $n\geq 1$  and $D\subset\subset X$ a relatively compact domain in $X$ such that both $H^{0,1}(X)$ and $H^{0,1}(D)$ are Hausdorff (in particular, it is true for $n\geq 2$, if $X$ and $D$ are either pseudoconvex or $1$-convex-$(n-1)$-concave). Then $D$ is a Runge domain in $X$ if and only if the natural map
$H^{n,n}_c(D)\to H^{n,n}_c(X)$ is injective.
\end{cor}

Using the characterization of pseudoconvexity in Stein manifolds by means of the Dolbeault cohomology with compact support, we get the following result.

\begin{cor}
Let $X$ be a Stein manifold of complex dimension $n\geq 2$  and $D\subset\subset X$ a relatively compact domain in $X$. Then  $D$ is pseudoconvex and Runge in $X$ if and only if $H^{n,r}_c(D)=0$, for $2\leq r\leq n-1$, and the natural map
$H^{n,n}_c(D)\to H^{n,n}_c(X)$ is injective.
\end{cor}
\begin{proof}
The necessary condition is a consequence of the characterization of pseudoconvex domains in Stein manifolds by means of the Dolbeault cohomology with compact support via the Serre duality and of Corollary \ref{0runge}.
For the sufficient condition, we have only to prove that the injectivity of the map $H^{n,n}_c(D)\to H^{n,n}_c(X)$ implies that $H^{n,n}_c(D)$ is Hausdorff. Let $f$ be a smooth $(n,n)$-form with compact support in $D$ such that $\int_D f\varphi=0$ for any holomorphic function $\varphi$ on $D$. In particular $\int_X f\varphi=0$ for any holomorphic function $\varphi$ on $X$ and  $X$ being Stein, $H^{n,n}_c(X)$ is Hausdorff and therefore $f=\opa u$ for some smooth $(n,n-1)$-form $u$ with compact support in $X$, i.e. $[f]=0$ in $H^{n,n}_c(X)$. By the injectivity of the map $H^{n,n}_c(D)\to H^{n,n}_c(X)$, we get that $f=\opa g$ for some smooth $(n,n-1)$-form $g$ with compact support in $D$, which ends the proof.
\end{proof}

\subsection{Some cohomological properties of the complement of a $q$-Runge domain}

Let us now relate the Runge property of the domain $D$ with some cohomological properties of $X\setminus D$. For any neighborhood $V$ of $X\setminus D$, we denote by $\Phi$ the family of supports in $V$, which consists of all closed subsets $F$ of $V$  such that $F\cup\ol D$ is a compact subset of $X$.
For $0\leq p,q\leq n$, we will say that
$H^{p,q}_\Phi(X\setminus D)=0$ if and only if for any neighborhood $V$ of $X\setminus D$ and for any $\opa$-closed $(p,q)$-form $f\in\ci_{p,q}(V)$ with ${\rm supp}~f\in\Phi$, there exist a neighborhood $U\subset V$ of $X\setminus D$ and a $(p,q-1)$-form $g\in\ci_{p,q-1}(U)$ with ${\rm supp}~g\subset F\in\Phi$ such that $\opa g=f$ on $U$.

\begin{thm}\label{complement}
Let $X$ be a non-compact complex manifold of complex dimension $n\geq 2$, $D\subset\subset X$ a relatively compact domain in $X$ and $q$ a fixed integer such that $0\leq q\leq n-2$. Assume  that, for any $0\leq p\leq n$, $H^{n-p,n-q}_c(X)$  and $H^{n-p,n-q}_c(D)$ are Hausdorff and $H^{n-p,n-q-1}_c(X)=0$. Then $D$ is a $q$-Runge domain in $X$ if and only if, for any $0\leq p\leq n$, $H^{n-p,n-q-1}_\Phi(X\setminus D)=0$.
\end{thm}

\begin{proof}
Assume $D$ is a $q$-Runge domain in $X$ and consider a neighborhood $V$ of $X\setminus D$ and a $\opa$-closed $(n-p,n-q-1)$-form $f\in\ci_{n-p,n-q-1}(V)$ with ${\rm supp}~ f\in\Phi$. Let $\chi$ be a positive smooth function with support in $V$ and equal to $1$ on a neighborhood $W$ of $X\setminus D$. Then $\chi f$ defines a form $\wt f$ such that $\opa\wt f$ has  compact support in $D$ and the cohomological class $[\opa\wt f]=0$ in $H^{n-p,n-q}_c(X)$. Then by Theorem \ref{runge}, $[\opa\wt f]=0$ in $H^{n-p,n-q}_c(D)$, which means that there exists a smooth $(n-p,n-q-1)$-form $u\in\ci_{n-p,n-q-1}(X)$ with compact support in $D$ such that $\opa u=\opa\wt f$. Set $h=\wt f-u$, then $h=f$ on a neighborhood $U$ of $X\setminus D$ and $\opa h=0$ on $X$. Since $H^{n-p,n-q-1}_c(X)=0$, there exists $g\in\ci_{n-p,n-q-2}(X)$ with compact support in $X$  such that $\opa g=h$ on $X$, which implies ${\rm supp}~g_{|U}\subset F\in\Phi$ and $\opa g_{|_U}=f$ on $U$.

Conversely we will prove that the natural map
$H^{n-p,n-q}_c(D)\to H^{n-p,n-q}_c(X)$ is injective, which, by Theorem \ref{runge}, implies that $D$ is a Runge domain in $X$. Let $f$ be a smooth $(n-p,n-q)$-form with compact support in $D$ such that $f=\opa g$ for a smooth $(n-p,n-q-1)$-form with compact support in $X$. Then $\opa g=0$ on some neighborhood $V$ of $X\setminus D$ and ${\rm supp}~g_{|_V}\in\Phi$ and by hypothesis there exist a neighborhood $U\subset V$ of $X\setminus D$ and an $(n-p,n-q-2)$-form $h\in\ci_{n-p,n-q-2}(U)$ with ${\rm supp}~h\subset F\in\Phi$ such that $\opa h=g$ on $U$. Let $\chi$ be positive smooth function with support in $U$ and equal to $1$ on an neighborhood $W$ of $X\setminus D$. Then $\chi h$ defines a smooth $(n-p,n-q-2)$-form $\wt h$ on $X$ with $\wt h=h$ on $W$ and if $u=g-\opa\wt h$, then $\opa u=\opa g=f$ and the support of $u$ is a compact subset of $D$.
\end{proof}

Note that the hypothesis $H^{n-p,n-q-1}_c(X)=0$ is only used to prove the necessary condition, i.e. if $D$ is a $q$-Runge domain in $X$, then  for any $0\leq p\leq n$, $H^{n-p,n-q-1}_\Phi(X\setminus D)=0$.

\begin{lem}\label{complementsupport}
Let $X$ be a Stein manifold of complex dimension $n\geq 2$ and $D$ a relatively compact domain in $X$. Then, for any $0\leq p\leq n$
and $1\leq q\leq n-2$, $H^{n-p,n-q-1}_\Phi(X\setminus D)=0$ if and only if $H^{n-p,n-q-1}(X\setminus D)=0$. Moreover if $H^{n-p,n-1}_\Phi(X\setminus D)=0$ for some $0\leq p\leq n$ then $H^{n-p,n-1}(X\setminus D)$ is Hausdorff.
\end{lem}
\begin{proof}
Let us prove the necessary condition in the first assertion.
Since $X$ is a Stein manifold, there exists a relatively compact strictly pseudoconvex set with $\cc^2$ boundary $U$ in $X$ such that $\ol D\subset U$.
It follows that $H^{p,q}(\ol U)=0$, for any $0\leq p\leq n$ and $1\leq q\leq n-1$, hence, by Proposition 1.1 in \cite{Lacalotte}, $H^{n-p,n-q-1}(X\setminus\ol U)=0$, for any $0\leq p\leq n$ and $1\leq q\leq n-2$, and $H^{n-p,n-1}(X\setminus\ol U)$ is Hausdorff, for any $0\leq p\leq n$.
First assume that $f$ is a smooth, $\opa$-closed $(n-p,n-q-1)$-form on a neighborhood of $X\setminus D$, then there exists a smooth $(n-p,n-q-2)$-form  $X\setminus\ol U$ such that $f=\opa g$ on $X\setminus\ol U$. Let  $\chi$ be a smooth function equal to $1$ on $X\setminus V$ for some neighborhood $V\subset\subset X$ of $\ol U$ and with support in $X\setminus\ol U$, then the support of $f-\opa(\chi g)$ belongs to $\Phi$. Since $H^{n-p,n-q-1}_\Phi(X\setminus D)=0$, $f-\opa\chi g=\opa u$ on a neighborhood of $X\setminus D$ and so $f=\opa(\chi g+ u)$.

Conversely, we will prove that the natural map $H^{n-p,n-q-1}_\Phi(X\setminus D)\to H^{n-p,n-q-1}(X\setminus D)$ is injective for any $0\leq p\leq n$ and $1\leq q\leq n-1$. Therefore $H^{n-p,n-q-1}(X\setminus D)=0$ will imply $H^{n-p,n-q-1}_\Phi(X\setminus D)=0$ for any $0\leq p\leq n$
and $1\leq q\leq n-2$. Let $f$ be a smooth, $\opa$-closed $(n-p,n-q-1)$-form on a neighborhood $V$ of $X\setminus D$ and whose support belongs to $\Phi$. Assume there exists a smooth $(n-p,n-q-2)$-form $g$ defined on a neighborhood $U\subset V$ of $X\setminus D$ and such that $f=\opa g$ on $U$. Consider a function $\chi$ {\color{red} with compact support in $X$} such that $\chi\equiv 1$ on a neighborhood of $\ol D\cup {\rm supp}f$. We set $\wt g=\chi g$. Then $\opa\wt g=\opa\chi\wedge g+\chi\opa g=\opa\chi\wedge g+f$ and the form $\opa\chi\wedge g$ can be extended by $0$ to a $\opa$-closed $(n-p,n-q-1)$-form with compact support in $X$. Since $X$ is Stein, there is an $h\in\dc^{n-p,n-q-2}(X)$such that $\opa h=\opa\chi\wedge g$ on $X$ and it follows that $\opa\wt g=\opa h+f$ on $U$. Then the support of  $u=\wt g-h$ belongs to $\Phi$ because $h$ has compact support in $X$ and $\opa u=f$ on $U$.

For the second assertion, let $f$ be a smooth, $\opa$-closed $(n-p,n-1)$-form on a neighborhood $V$ of $X\setminus D$ such that $\int f\wedge\varphi=0$ for any $\opa$-closed smooth $(p,1)$-form with compact support in a neighborhood $U\subset V$ of $X\setminus D$. Since $H^{n-p,n-1}(X\setminus\ol U)$ is Hausdorff and $X\setminus U\subset X\setminus D$, there exists a smooth $(n-p,n-q-1)$-form on $X\setminus\ol U$ such that $f=\opa g$ on $X\setminus\ol U$. Then we can repeat the proof as before for the necessary condition.
\end{proof}

\begin{prop}\label{pseudoconvexity}
Let $X$ be a Stein manifold of complex dimension $n\geq 2$ and $D$ a relatively compact domain in $X$. Then, for any  $1\leq q\leq n-2$, $H^{n,q}_c(D)=0$ if and only if $H^{n,q-1}(X\setminus D)=0$. Moreover if $H^{n,n-1}_\Phi(X\setminus D)=0$, then $H^{n,n}_c(D)$ is Hausdorff.
\end{prop}
\begin{proof}
Assume $H^{n,q}_c(D)=0$ and consider a neighborhood $V$ of $X\setminus D$ and a $\opa$-closed $(n,q-1)$-form $f\in\ci_{n,q-1}(V)$. Let $\chi$ be a positive smooth function with support in $V$ and equal to $1$ on a neighborhood $W$ of $X\setminus D$. Then $\chi f$ defines a form $\wt f$ such that $\opa\wt f$ has  compact support in $D$. Since $H^{n,q}_c(D)=0$, there exists a smooth $(n-p,n-q-1)$-form $u\in\ci_{n-p,n-q-1}(X)$ with compact support in $D$ such that $\opa u=\opa\wt f$. Set $h=\wt f-u$, then $h=f$ on a neighborhood $U$ of $X\setminus D$ and $\opa h=0$ on $X$. Since $X$ is Stein, there exists $g\in\ci_{n-p,n-q-2}(X)$ such that $\opa g=h$ on $X$, which implies  $\opa g_{|_U}=f$ on $U$.

Conversely assume $H^{n,q-1}(X\setminus D)=0$. Let $f$ be a smooth $(n,q)$-form with compact support in $D$, then $f=\opa g$ for a smooth $(n,q-1)$-form $g$ with compact support in $X$, since $X$ is Stein. Then $\opa g=0$ on some neighborhood $V$ of $X\setminus D$  and by hypothesis there exists a neighborhood $U\subset V$ of $X\setminus D$ and an $(n-p,n-q-2)$-form $h\in\ci_{n-p,n-q-2}(U)$ such that $\opa h=g$ on $U$. Let $\chi$ be positive smooth function with support in $U$ and equal to $1$ on an neighborhood $W$ of $X\setminus D$. Then $\chi h$ defines a smooth $(n-p,n-q-2)$-form $\wt h$ on $X$ with $\wt h=h$ on $W$ and if $u=g-\opa\wt h$, then $\opa u=\opa g=f$ and the support of $u$ is a compact subset of $D$.

Assume $H^{n,n-1}_\Phi(X\setminus D)=0$. Let $f$ be a smooth $(n,n)$-form with compact support in $D$, which is orthogonal to any holomorphic function on $D$. In particular $\int_X f\varphi=0$ for any holomorphic function $\varphi$ on $X$ and $X$ being Stein, $H^{n,n}_c(X)$ is Hausdorff and therefore $f=\opa g$ for some smooth $(n,n-1)$-form $g$ with compact support in $X$. Then $\opa g=0$ on some neighborhood $V$ of $X\setminus D$ and the support of $g_{|_V}$ belongs to $\Phi$. By hypothesis there exists a neighborhood $U\subset V$ of $X\setminus D$ and an $(n,n-2)$-form $h\in\ci_{n,n-2}(U)$ whose support is included in some $F$ which  belongs to $\Phi$ and such that $\opa h=g$ on $U$. Repeating the same arguments, the proposition is proved.
\end{proof}

The next corollary follows directly from Theorem \ref{complement}, Proposition \ref{pseudoconvexity}, Lemma \ref{complementsupport} and the characterization of pseudoconvexity in Stein manifolds by means of the Dolbeault cohomology with compact support.

\begin{cor}\label{caracrunge}
Let $X$ be a Stein manifold of complex dimension $n\geq 2$  and $D\subset\subset X$ a relatively compact domain in $X$ such that $X\setminus D$ is connected. Then  $D$ is pseudoconvex and Runge in $X$ if and only if $H^{n,q}_\Phi(X\setminus D)=0$ for all $1\leq q\leq n-1$.
\end{cor}

Let us end with geometric conditions to ensure the Runge density properties.

We say that the manifold $X$ is an $r$-convex extension of a domain $\Omega\subset X$ if the boundary of $\Omega$ is compact and there exists a $\cc^2$ real valued function $\rho$ defined on a neighborhood $U$ of $X\setminus\Omega$, whose Levi form admits at least $(n-r+1)$ positive eigenvalues, such that $\Omega\cap U=\{z\in U~|~\rho(z)<0\}$ and for any real number $0<c<\sup_{z\in U} \rho(z)$, the set $\{z\in U~|~0\leq\rho(z)\leq c\}$ is compact.

Using Theorem 16.1 in \cite{HeLe2}, the next corollary follows from Theorem \ref{complement} since in a $(q+1)$-convex manifold $M$, for any $0\leq p\leq n$, $H^{n-p,n-q}_c(M)$ is Hausdorff .

\begin{cor}\label{extension}
Let $X$ be a non-compact complex manifold of complex dimension $n\geq 2$, $D\subset\subset X$ a relatively compact domain in $X$ and $q$ a fixed integer such that $0\leq q\leq n-2$. Assume  that, for any neighborhood $V$ of $X\setminus D$, there exists a domain $\Omega$ such that $X\setminus V\subset\Omega\subset D$ and $X$ is a $(q+1)$-convex extension of $\Omega$. Then $D$ is a $q$-Runge domain in $X$.
\end{cor}

Let us consider the special case when $p=q=0$. If $X$ is a Stein manifold of complex dimension $n\geq 2$, it follows from the Hartogs extension phenomenon for holomorphic functions that we have only to consider domains $D$ such that $X\setminus D$ has no relatively compact connected component, i.e. $X\setminus D$ is connected. If $X$ is $1$-convex-$(n-1)$-concave, we can consider independently the connected component of $X\setminus D$ which contain the $1$-convex end of $X$ and the connected components of $X\setminus D$ which contain the $(n-1)$-concave ends of $X$.

\begin{cor}
Let $X$ be a complex manifold of complex dimension $n\geq 2$ and $D\subset\subset X$ a relatively compact domain in $X$ such that $X\setminus D$ has no relatively compact connected component. For any connected component $D^c$ of $X\setminus D$, assume that there exists a neighborhood $V$ of $D^c$, which does not meet any other connected component of $X\setminus D$, and a domain $\Omega$ such that $D^c\subset X\setminus\Omega\subset V$ and either $V$ is a $1$-convex extension of $\Omega$ or an $(n-1)$-concave extension of $\Omega$, then $D$ is a Runge domain in $X$.
\end{cor}
\begin{proof}
The hypothesis implies that $X$ is either a $1$-convex or a $1$-convex-$(n-1)$-concave complex manifold and hence $H^{n,n}_c(X)$ is Hausdorff.

Under the assumptions on the connected components of $X\setminus D$ which contain the $(n-1)$-concave ends of $X$, any holomorphic function on $\wt V\cap D$, where $\wt V$ is a neighborhood of such a component, extends holomorphically to $\wt V$ (see the last section of \cite{LaLeMathAnn}). For the connected component $D^c$ which contain  the $1$-convex end of $X$, we can apply Theorem 16.1 in \cite{HeLe2} to get that $H^{n,q}_\Phi(D^c)=0$. The conclusion follows then from Theorem \ref{complement}.
\end{proof}

\subsection{Runge density property for germs of holomorphic functions on a compact set}

\begin{defin}
A compact subset $K$ of a non-compact complex manifold $X$ of complex dimension $n\geq 1$ is  Runge  in $X$ if and only if the space $\oc(X)$ of holomorphic functions on $X$ is dense in the space $\oc(K)$ of holomorphic functions in a neighborhood of $K$ for the topology of $\ci(K)$.
\end{defin}

Let $K$ be a compact subset in a complex manifold $X$ and $q$ a positive integer, we will say that $H^{0,q}(K)=0$ if any $\opa$-closed $(0,q)$-current defined on a neighborhood of $K$ is $\opa$-exact on a possibly smaller neighborhood of $K$.

\begin{thm}\label{rungeK}
Let  $X$ be a Stein manifold of complex dimension $n\geq 1$ and  $K$ a compact subset of $X$ with connected complement. Assume that for any  $V$ belonging to a neighborhood basis of $K$ the natural map $H^{n,n}_c(V)\cap (\ec'_K)^{n,n}(X)\to H^{n,n}_c(X)$ is injective, then $K$ is  Runge  in $X$. If moreover $n\geq 2$, $H^{0,1}(K)=0$ and $K$ is Runge in $X$, then, for any neighborhood $V$ of $K$, the natural map $H^{n,n}_c(V)\cap (\ec'_K)^{n,n}(X)\to H^{n,n}_c(X)$ is injective.
\end{thm}
\begin{proof}
We will use the Hahn-Banach theorem. Let $T$ be an $(n,n)$-current with support in $K$ such that $<T,f>=0$ for any holomorphic function $f\in\oc(X)$. Since $H^{n,n}_c(X)$ is Hausdorff, there exists an $(n,n-1)$-current $S$ with compact support in $X$ such that $T=\opa S$. Let $\varphi\in\oc(K)$, then there exists a neighborhood $V$ of $K$ such that $\varphi\in\oc(V)$. Using the injectivity hypothesis there exists an $(n,n-1)$-current $R$ with compact support in $V$ such that $T=\opa R$, therefore
$$<T,\varphi>=<\opa R,\varphi>=\pm <R,\opa\varphi>=0.$$
By the Hahn-Banach theorem, we get the density property.

Now assume $H^{0,1}(K)=0$. Let $T\in(\ec'_K)^{n,n}(X)$ such that $T=\opa S$ for an $(n,n-1)$-current $S$ with compact support in $X$. We first prove that for any $\varphi\in\oc(K)$, we have $<T, \varphi>=0$. Since $\oc(X)$  is dense in the space $\oc(K)$  for the topology of $\ci(K)$, there exists a sequence $(\varphi_k)_{k\in\nb}$ of holomorphic functions in $X$ which converges for the smooth topology on $K$ to $\varphi$. So
$$<T, \varphi>=\lim_{k\to\infty}<T, \varphi_k>=\lim_{k\to\infty}<\opa S, \varphi_k>=\pm\lim_{k\to\infty}<S,\opa \varphi_k>=0.$$

Since the support of $T$ is contained in $K$, the current $S$ is $\opa$-closed in $X\setminus K$. Recall that if $H^{0,1}(K)=0$, then $H^{n,n-1}(X\setminus K)$ is Hausdorff (see Proposition 1.1 in \cite{Lacalotte}). Therefore to prove that $S$ is $\opa$-exact in $X\setminus K$, it is sufficient to prove that for any smooth, $\opa$-closed $(0,1)$-form $\theta$ with compact support in $X\setminus K$, we have $<S,\theta>=0$. Using that $X$ is Stein $\theta=\opa\omega$ for some smooth function $\omega$ with compact support in $X$, moreover $\omega$ is holomorphic in some neighborhood of $K$. So
$$<S,\theta>=<S,\opa\omega>=\pm <\opa S,\omega>=<T,\omega>=0$$
and $S=\opa R$ for an $(n,n-2)$-current $R$ in $X\setminus K$.

Let $V$ be some neighborhood of $K$ and $\chi$ be a smooth positive function on $X$ such that $\chi\equiv 1$ on $X\setminus V$ and vanishing on a neighborhood of $K$. Set $S_V=S-\opa\chi R$, then $S_V$ has compact support in $V$ and $T=\opa S_V$.
\end{proof}

Let $K$ be a compact subset of a complex manifold of complex dimension $n$. Let us denote by $\Phi$ the family of all closed subset $F$ of $X\setminus K$ such that $F\cup K$ is a compact subset of $X$. We will say that $\wt H^{p,q}_\Phi(X\setminus K)=0$ for some $0\leq p\leq n$ and $1\leq q\leq n-1$ if and only if for any extendable $\opa$-closed $(p,q)$-current $T$ on $X\setminus K$, whose support belongs to $\Phi$, there exists a $(p,q-1)$-current $S$ on $X\setminus K$, whose support belongs to $\Phi$, such that $T=\opa S$ on $X\setminus K$.

\begin{thm}
Let $X$ be a Stein manifold of complex dimension $n\geq 2$ and $K$ a compact subset of $X$. Assume $\wt H^{n,n-1}_\Phi(X\setminus K)=0$, then for any   neighborhood $V$ of $K$ the natural map $H^{n,n}_c(V)\cap (\ec'_K)^{n,n}(X)\to H^{n,n}_c(X)$ is injective.
\end{thm}
\begin{proof}
First let $T$ be an $(n,n)$-current on $X$ with support contained in $K$ such that $T=\opa S$ for some $(n,n-1)$-current $S$  with compact support in $X$. Then $S_{|_{X\setminus K}}$ is an extendable $\opa$-closed $(n,n-1)$-current on $X\setminus K$, whose support belongs to $\Phi$. Since $\wt H^{n,n-1}_\Phi(X\setminus K)=0$, there exists an $(n,n-2)$-current $U$  on $X\setminus K$, whose support belongs to $\Phi$, such that $S=\opa U$. Let $V$ be a neighborhood of $K$ and $\chi$ a positive smooth function with support in $X\setminus K$ and equal to $1$ on a neighborhood of $X\setminus V$. Set $\wt U=\chi U$, then $S-\opa\wt U$ has compact support in $V$ and $\opa(S-\opa\wt U)=T$.
\end{proof}

\begin{lem}\label{lemmecompl}
Let $X$ be a non-compact complex manifold  $X$ of complex dimension $n\geq 2$, $K$ a compact subset of $X$, $U$ a relatively compact neighborhood of $K$. We assume that

i) $H^{n,n-1}_c(X)=0$,

ii) the natural map $H^{n,n}_c(U)\cap (\ec'_K)^{n,n}(X)\to H^{n,n}_c(X)$ is injective.

\parindent=0pt{For any extendable, $\opa$-closed $(n,n-1)$-current $T$ on $X\setminus K$ vanishing outside a compact subset of $X$, there is an $(n,n-2)$-current $S$ on $X\setminus\ol U$ vanishing outside a compact subset of $X$ such that $\opa S=T$ on $X\setminus\ol U$.}
\end{lem}
\begin{proof}
 Let $T$ is an extendable, $\opa$-closed $(n,n-1)$-current on $X\setminus K$ vanishing outside a compact subset of $X$ and $\wt T$ an extension of $T$ to $X$, then $\wt T$ defines an $(n,n-1)$-current with compact support in $X$ and the support of $\opa\wt T$ is contained in $K$. So by ii) there is an  $(n,n-1)$-current $R$ with compact support in $U$ such that $\opa R=\opa\wt T$ on $X$. The current $\wt T-R$ is $\opa$-closed and compactly supported in $X$. Hypothesis i) then implies the existence of an $(n,n-2)$-current $S$ with compact support in $X$ such that $\wt T-R=\opa S$. The restriction of $S$ to $X\setminus\ol U$ is then the form we seek because $\wt T-R=T$ on $X\setminus\ol U$.
\end{proof}

\begin{thm}\label{complementK}
Let  $X$ be a Stein manifold of complex dimension $n\geq 2$, $K$ a compact subset of $X$. We assume that $K$ admits a decreasing Stein neighborhood basis $(U_k)_{k\in\nb}$ such that $\cap_{k\in\nb}U_k=K$ and, for any $k\in\nb$, $X\setminus\ol U_k$ is connected and  $H^{n,n}_c(U_k)\cap (\ec'_K)^{n,n}(X)\to H^{n,n}_c(X)$ is injective. Then $\wt H^{n,n-1}_\Phi(X\setminus K)=0$.
\end{thm}
\begin{proof}
Let $T$ be an extendable, $\opa$-closed $(n,n-1)$-current on $X\setminus K$, whose support belongs to $\Phi$.
Since $X$ is a Stein manifold, the hypotheses of Lemma \ref{lemmecompl} are fulfilled. Hence for each $k\in\nb$, there exists an $(n,n-2)$-current $S_k$ on $X\setminus\ol U_k$ vanishing outside a compact subset of $X$ such that $\opa S_k=T$ on $X\setminus\ol U_k$.

If $n=2$, the distribution $S_{k+1}-S_k$ is then holomorphic on $X\setminus\ol U_k$ and vanishes outside a compact subset of $X$. By analytic continuation, $X\setminus\ol U_k$ being connected, we get $S_{k+1}-S_k\equiv 0$ on $X\setminus\ol U_k$. The distribution $S$ defined by $S=S_k$ on $X\setminus\ol U_k$ satisfies ${\rm supp}~S\in\Phi$ and $\opa S=T$ on $X\setminus K$.

Now we suppose $n\geq 3$. We proceed by induction. We set $\wt S_0=S_2{_{|{X\setminus\ol U_0}}}$ and we assume that, for $1\leq k\leq l$,  we have already construct $\wt S_k$ vanishing outside a compact subset of $X$ and such that $\opa\wt S_k=T$ on $X\setminus\ol U_{k+2}$ and $\wt S_k{_{|_{X\setminus\ol U_{k-1}}}}=\wt S_{k-1}$. We construct $\wt S_{l+1}$ in the following way. The current $\wt S_l- S_{l+3}$ is $\opa$-closed on $X\setminus\ol U_{l+2}$. Without loss of generality we may assume that each $U_k$ is a strictly pseudoconvex domain with $\cc^2$ boundary and $U_{k+1}\subset\subset U_k$. The strict pseudoconvexity of $U_k$ implies that $H^{0,1}(\ol U_k)=0$ and by Proposition 1.1 in \cite{Lacalotte}, $H^{n,n-2}(X\setminus \ol U_k)=0$. Therefore there exists $R$ such that $\opa R= \wt S_l- S_{l+3}$ on $X\setminus\ol U_{l+1}$. Let $\chi$ be a smooth function on $X$ such that $\chi\equiv 1$ on a neighborhood of $X\setminus\ol U_{l}$ and $\chi\equiv 0$ on a neighborhood of $\ol U_{l+1}$. We set $\wt S_{l+1}=S_{l+3}+\opa\chi R$, then $\opa\wt S_{l+1}=T$ on $X\setminus\ol U_{l+3}$ and $\wt S_{l+1}{_{|_{X\setminus\ol U_{l}}}}=\wt S_{l}$. The current $S$ on $X\setminus K$ defined by $S=\wt S_k$ on $X\setminus\ol U_k$ satisfies ${\rm supp}~S\in\Phi$ and $\opa S=T$ on $X\setminus K$.
\end{proof}

A compact subset $K$ in a complex manifold $X$ is called \emph{holomorphically convex} if and only if $H^{0,q}(K)=0$ for any $1\leq q\leq n-1$ and is a \emph{Stein compactum} if and only if it admits a Stein neighborhood basis. Note that any Stein compactum is clearly holomorphically convex.
We deduce from the previous theorems the following characterization of Runge compact subset of a Stein manifold.

\begin{cor}\label{rungecompact}
Let $X$ be a Stein manifold of complex dimension $n\geq 2$ and $K$ a holomorphically convex subset of $X$. Consider the following assertions:

(i) $K$ is Runge in $X$;

(ii) for any neighborhood $V$ of $K$, the natural map $H^{n,n}_c(V)\cap (\ec'_K)^{n,n}(X)\to H^{n,n}_c(X)$ is injective.

(iii) $\wt H^{n,n-1}_\Phi(X\setminus K)=0$.

{\parindent=0pt Then (i) is equivalent to (ii) and (iii) implies (ii). If moreover $K$ is a Stein compactum, then (ii) is equivalent to (iii).}
\end{cor}

To end this section, we will give some sufficient conditions on the compact subset $K$ to ensure the stronger condition $H^{n,n-1}_\Phi(X\setminus K)=0$ to hold. The Dolbeault cohomology groups $H^{p,q}_\Phi(X\setminus K)=0$ are directly related to the study of removable singularities for CR forms or functions defined on a part of the boundary of a domain (see \cite{Lu}, \cite{LaLeMathAnn}, \cite{ChSt}, \cite{Lalivre}).

In the same way we proved Theorem \ref{complementK}, we get the following result.

\begin{thm}\label{complementKfort}
Let  $X$ be a Stein manifold of complex dimension $n\geq 2$, $K$ a compact subset of $X$ and $q$ a fixed integer such that $0\leq q\leq n-2$. We assume that $K$ admits a decreasing neighborhood basis $(U_k)_{k\in\nb}$ consisting of $(q+1)$-convex $q$-Runge domains such that $\cap_{k\in\nb}U_k=K$. Then $H^{n-p,n-q-1}_\Phi(X\setminus K)=0$, for any $0\leq p\leq n$.
\end{thm}

\begin{rem}\label{O(X)convex}
Let us notice that if $p=q=0$, then the hypothesis in Theorem \ref{complementKfort} becomes: there exists a decreasing neighborhood basis $(U_k)_{k\in\nb}$ of $K$ consisting of pseudoconvex  domains which are Runge in $X$ such that $\cap_{k\in\nb}U_k=K$. But this property characterizes the compact subsets of the Stein manifold $X$ which are $\oc(X)$-convex, i.e. $K=\widehat K_X$, where $\widehat K_X=\{z\in X~|~\forall f\in\oc(X), |f(z)|\leq \sup_{\zeta\in K}|f(\zeta)|\}$. In fact if $K$ is $\oc(X)$-convex, then, by Theorem 8.17 from Chapter VII in \cite{Lalivre}, $K$ admits  a decreasing neighborhood basis $(U_k)_{k\in\nb}$ consisting of pseudoconvex  domains which are Runge in $X$. If $U$ is a pseudoconvex neighborhood of $K$ which is a Runge domain in $X$, then
$$\widehat K_X=\widehat K_U\subset\subset U,$$
which proves the converse.
\end{rem}

Following this remark, as a corollary of Corollary \ref{rungecompact} and Theorem \ref{complementKfort} we recover the Oka-Weil theorem.

\begin{cor}\label{okaweil}
Let $X$ be a Stein manifold of complex dimension $n\geq 2$ and  $K$ a compact subset of $X$. Assume $K$ is $\oc(X)$-convex, then $K$ is Runge in $X$.
\end{cor}

Moreover using once again Theorem 8.17 from Chapter VII in \cite{Lalivre} and Corollary \ref{extension}, we can also get that if $K$ is an $\oc(X)$-convex compact subset of the Stein manifold $X$, the hypothesis of Theorem \ref{complementKfort} is fulfilled for all $0\leq q\leq n-2$.

\section{Some new Runge density properties}

Let $X$ be a complex manifold of complex dimension $n\geq 1$ and  $D\subset\subset X$ a relatively compact domain in $X$.

In this section we will always assume that the boundary of $D$ is Lipschitz to be able to use Serre duality. Lipschitz boundary ensures that the space of $(p,q)$-currents in $D$, which extend as a current to $X$ and the space $\dc^{n-p,n-q}_{\ol D}(X)$ of smooth $(n-p,n-q)$-forms with support contained in the closure of $D$ are dual to each other and the space $\ci_{p,q}(\ol D)$ of smooth $(p,q)$-forms in $\ol D$ and the space ${\ec'}^{n-p,n-q}_{\ol D}(X)$ of $(n-p,n-q)$-currents with support contained in the closure of $D$ are dual to each other (see \cite{LaShdualiteL2}).

Moreover, let us consider the densely defined operators $\opa$ from $L^2_{p,q}(D)$ into $L^2_{p,q+1}(D)$ whose domain is the subspace of $(p,q)$-forms $f$ such that $f\in L^2_{p,q}(D)$ and $\opa f\in L^2_{p,q+1}(D)$ and $\opa_{\wt c}$ from $L^2_{p,q}(D)$ into $L^2_{p,q+1}(D)$ whose domain is the subspace of $(p,q)$-forms $f$ such that $f\in L^2_{p,q}(X)$, ${\rm supp}~f\subset\ol D$ and $\opa f\in L^2_{p,q+1}(X)$. If the boundary of $D$ is Lipschitz, then the associated complexes are dual to each other (see lemma 2.4 in \cite{LaShdualiteL2}).

\subsection{The $\ci$-Runge density property and the $\ci$-Mergelyan property}

\begin{defin}
A relatively compact domain $D$ in $X$ is \emph{$\ci$ $q$-Runge} in $X$, for $0\leq q\leq n-1$, if and only if, for any $0\leq p\leq n$, the space $Z^{p,q}_\infty(X)$ of smooth $\opa$-closed $(p,q)$-forms in $X$ is dense in the space $Z^{p,q}_\infty(\ol D)$ of  $\opa$-closed $(p,q)$-forms in $D$, which  are smooth on the closure of $D$, for the smooth topology  on the closure of $D$.

For $q=0$, we will simply say that the domain is \emph{$\ci$-Runge} in $X$, which means that the space $\oc(X)$ of holomorphic functions in $X$ is dense in the space $\oc(D)\cap\ci(\ol D)$ of holomorphic functions in $D$, which are smooth on the closure of $D$, for the smooth topology  on the closure of $D$.

\end{defin}

If $D\subset\subset X$ is a relatively compact domain in $X$, we denote by $H^{r,s}_{\ol D, cur}(X)$ the  Dolbeault cohomology groups of currents with prescribed support in $\ol D$ and by $\check{H}^{r,s}_{\Phi}(X\setminus\ol D)$ the Dolbeault cohomology groups of extendable currents in $X\setminus\ol D$ vanishing outside a compact subset of $X$.

\begin{thm}\label{stronglyrunge}
Let $X$ be a non-compact complex manifold of complex dimension $n\geq 1$, $D\subset\subset X$ a relatively compact domain with Lipschitz boundary in $X$ and $q$ a fixed integer such that $0\leq q\leq n-1$. Assume  that, for any $0\leq p\leq n$, $H^{n-p,n-q}_c(X)$ and  $H^{n-p,n-q}_{\ol D,cur}(X)$ are Hausdorff. Then $D$ is a $\ci$ $q$-Runge domain in $X$ if and only if the natural map
$H^{n-p,n-q}_{\ol D,cur}(X)\to H^{n-p,n-q}_c(X)$
is injective.
\end{thm}
\begin{proof}
Since $D$ has a Lipschitz boundary, the space $\ci_{p,q}(\ol D)$ of smooth functions in $\ol D$ and the space ${\ec'}^{n-p,n-q}_{\ol D}(X)$ of $(n-p,n-q)$-currents with support contained in the closure of $D$ are dual to each other, consequently the proof follows the same arguments as in the proof of Theorem \ref{runge}.
\end{proof}

\begin{prop}\label{strongcomplement}
Let $X$ be a non-compact complex manifold of complex dimension $n\geq 2$, $D\subset\subset X$ a relatively compact domain in $X$ and $q$ be a fixed integer such that $0\leq q\leq n-2$. Assume  that, for some $0\leq p\leq n$,  $H^{n-p,n-q-1}_c(X)=0$.
 Then $\check{H}^{n-p,n-q-1}_{\Phi}(X\setminus\ol D)=0$ if and only if the natural map
$H^{n-p,n-q}_{\ol D,cur}(X)\to H^{n-p,n-q}_c(X)$
is injective.
\end{prop}

\begin{proof}
We first consider the necessary condition. Let $T\in(\ec')^{n-p,n-q}_{\ol D}(X)$ such that $T=\opa S$ with $S\in(\ec')^{n-p,n-q-1}(X)$. Since the support of $T$ is contained in $\ol D$, we have $\opa S=0$ on $X\setminus\ol D$. Therefore the vanishing of the group $\check{H}^{n-p,n-q-1}_{c,\infty}(X\setminus D)$ implies that there exists $U\in(\check{\dc'})^{n-p,n-q-2}(X\setminus\ol D)$ such that $\opa U=S$ on $X\setminus\ol D$. Let $\wt U$ be an extension of $U$ to $X$, we set $R=S-\opa\wt U$, then  $R$ is a current on $X$, $T=\opa R$ and ${\rm supp}~R\subset\ol D$.

Conversely, let $S$ be a $\opa$-closed, extendable $(n-p,n-q-1)$-current on $X\setminus D$ with compact support and $\wt S$ a smooth extension of $S$ to $X$, then $\wt S$ has compact support in $X$ and $T=\opa\wt S$ is an element of $(\ec')^{n-p,n-q}_{\ol D}(X)$. By the injectivity of the natural map
$H^{n-p,n-q}_{\ol D,cur}(X)\to H^{n-p,n-q}_c(X)$, there exists $U\in (\ec')^{n-p,n-q-1}_{\ol D}(X)$ such that $\opa U=T$. We set $R=\wt S-U$, $R$ is then a smooth $\opa$-closed $(n-p,n-q-1)$-current with compact support in $X$ such that $R_{|_{X\setminus D}}=S$. Since $H^{n-p,n-q-1}_c(X)=0$, we have $R=\opa W$ with $W$ with compact support in $X$. Finally we get $S=R_{|_{X\setminus\ol D}}=\opa W_{|_{X\setminus\ol D}}$.
\end{proof}

Note that the hypothesis $H^{n-p,n-q-1}_c(X)=0$ is only used to prove the sufficient condition, i.e. if the natural map
$H^{n-p,n-q}_{\ol D,cur}(X)\to H^{n-p,n-q}_c(X)$
is injective, then $\check{H}^{n-p,n-q-1}_{\Phi}(X\setminus\ol D)=0$.

In the spirit of Corollary \ref{extension}, we can derive the next corollary from Theorem 16.1 in \cite{HeLe2} and Theorem 4 in \cite{Saext}.

\begin{cor}\label{strongextension}
Let $X$ be a non-compact complex manifold of complex dimension $n\geq 2$, $D\subset\subset X$ a relatively compact domain in $X$ with smooth boundary and $q$ a fixed integer such that $0\leq q\leq n-2$. Assume  that $X$ is a $(q+1)$-convex extension of $\ol D$. Then $D$ is a $\ci$ $q$-Runge domain in $X$.

In particular, if $p=q=0$, $n\geq 2$ and $X$ is a $1$-convex extension of $\ol D$, then $D$ is $\ci$-Runge in $X$.
\end{cor}

In the case when $q=0$, we derive the following result:

\begin{cor}\label{mergelyan}
Let $X$ be a Stein  manifold of complex dimension $n\geq 2$ and $D\subset\subset X$ a relatively compact pseudoconvex domain with Lipschitz boundary in $X$. Consider the following assertions:

(i) the domain $D$ is $\ci$-Runge in $X$,

(ii) the natural map
$H^{n,n}_{\ol D,cur}(X)\to H^{n,n}_c(X)$
is injective,

(iii) $\check{H}^{n,n-1}_{\Phi}(X\setminus\ol D)=0$,

{\parindent=0pt Then (iii) is equivalent to (ii) and (ii) implies (i). If moreover $D$ has smooth boundary, then (i) implies (ii).}
\end{cor}
\begin{proof}
Since $X$ is Stein, we have $H^{n,n-1}_c(X)=0$ and $H^{n,n}_c(X)$ is Hausdorff. The domain $D$ being relatively compact, pseudoconvex, with smooth boundary in $X$,  we have $H^{0,1}_\infty(\ol D)=0$ by Kohn's classical result on solving $\opa$ smoothly up to the boundary in pseudoconvex domains with smooth boundary.
Then the  Serre duality implies that $H^{n,n}_{\ol D, cur}(X)$ is Hausdorff.
The corollary follows then from Theorem \ref{stronglyrunge} and Proposition \ref{strongcomplement}.
\end{proof}

\begin{cor}
Let $X$ be a Stein  manifold of complex dimension $n\geq 2$ and $D\subset\subset X$ a relatively compact domain with smooth boundary in $X$ such that $X\setminus D$ is connected. Then $D$ is pseudoconvex and $\ci$-Runge in $X$ if and only if $\check{H}^{n,q}_{\Phi}(X\setminus\ol D)=0$ for all $1\leq q\leq n-1$.
\end{cor}
\begin{proof}
Following the arguments in the proof of Lemma \ref{complementsupport}, we can prove that if $\check{H}^{n,q}_{\Phi}(X\setminus\ol D)=0$ for all $\leq q\leq n-1$, then $\check{H}^{n,q}(X\setminus D)=0$, for any  $1\leq q\leq n-2$ and $\check{H}^{n,n-1}(X\setminus D)=0$ is Hausdorff.
Moreover, Theorem 3.11 in \cite{FuLaSh} implies that $H^{n,q}_{\ol D, cur}(X)=0$ if and only if $\check{H}^{n,q-1}(X\setminus D)=0$, for any  $1\leq q\leq n-1$. Proposition 3.7 in \cite{FuLaSh} implies that if $\check{H}^{n,n-1}(X\setminus D)=0$, then $H^{n,n}_{\ol D, cur}(X)$ is Hausdorff.
Therefore the corollary follows from Corollary 3.12, Theorem 3.13 in \cite{FuLaSh} and Corollary \ref{mergelyan}.
\end{proof}

\begin{defin}
A relatively compact domain $D$ in $X$ has the \emph{$\ci$ $q$-Mergelyan property}, for $0\leq q\leq n-1$, if and only if, for any $0\leq p\leq n$, any form in the space $Z^{p,q}_\infty(\ol D)$ of smooth $\opa$-closed $(p,q)$-forms in $\ol D$ can be approximated, for the smooth topology  on the closure of $D$, by smooth $\opa$-closed forms defined on a neighborhood of $\ol D$.

If $p=q=0$, this means the space $\oc(\ol D)$ of germs of holomorphic functions on $\ol D$ is dense in the space $\oc(D)\cap\ci(\ol D)$ of holomorphic functions on $D$, which are smooth on $\ol D$, for the smooth topology on the closure of $\ol D$. In that case we will say that $D$ has the \emph{$\ci$-Mergelyan property}.
\end{defin}

Note that, for a relatively compact domain $D$ in a complex manifold $X$ to have the $\ci$ $q$-Mergelyan property, it is sufficient that $D$ admits a neighborhood $V$ such that $D$ is $\ci$ $q$-Runge in $V$. So it comes from Corollary \ref{strongextension} that

\begin{prop}\label{spseudoconvex}
Let $D$ be a relatively compact domain in a complex manifold $X$ of complex dimension $n\geq 2$. Assume $D$ has smooth boundary and $\ol D$ admits a neighborhood $V$ which is a $(q+1)$-convex extension of $\ol D$, then $D$ has the $\ci$ $q$-Mergelyan property.
\end{prop}

If $q=0$, Proposition \ref{spseudoconvex} could be compare to Theorem 24 in the survey paper \cite{FoForWo}, where the authors study the $\cc^k$-Mergelyan property. They assume the domain $D$ to be strictly pseudoconvex in a Stein manifold. Here we only need the domain $D$ to be extendable in a $1$-convex way to some $1$-convex neighborhood.
\medskip

To end this section, let us relate the $\ci$-Mergelyan property with the solvability of the $\opa$-equation with prescribed support.

\begin{thm}\label{mergelyanprop}
Let $X$ be a complex manifold of complex dimension $n\geq 1$, $D\subset\subset X$ a relatively compact domain with Lipschitz boundary in $X$. Assume  $\ol D$ admits a neighborhood basis of $1$-convex open subsets. Assume that for all $(n,n)$-currents $T$ with support in $\ol D$ such that, for any sufficiently small neighborhood $V$ of $\ol D$, $T=\opa S_V$ for some $(n,n-1)$-current $S_V$ with compact support in $V$, there exists an $(n,n-1)$-current $S$ with support in $\ol D$ satisfying $T=\opa S$, then $D$ has the $\ci$-Mergelyan property.
\end{thm}
\begin{proof}
Assume the cohomological condition holds. Let $T$ be an $(n,n)$-current with support in $\ol D$ such that $<T,f>=0$ for any $f\in\oc(\ol D)$. For any $1$-convex neighborhood $V$ of $\ol D$, $H^{n,n}_c(V)$ is Hausdorff. Hence there exists an $(n,n-1)$-current $S_V$ with compact support in $V$ such that $T=\opa S_V$. Using the hypothesis, we get that $T=\opa S$ for some $(n,n-1)$-current $S$ with support in $\ol D$. Let $g\in\oc(D)\cap\ci(\ol D)$, then
$$<T,g>=<\opa S,g>=\pm <S,\opa g>=0.$$
We apply now the Hahn-Banach theorem to get the density property.
\end{proof}

Conversely we get the next theorem.

\begin{thm}\label{mergelyancompact}
Let $X$ be a complex manifold of complex dimension $n\geq 1$, $D\subset\subset X$ a relatively compact domain with Lipschitz boundary in $X$. Assume $H^{n,n}_{\ol D,cur}(X)$ is Hausdorff and $D$ has the $\ci$-Mergelyan property, then if  $T$ is an $(n,n)$-current with support in $\ol D$ such that, for any neighborhood $V$ of $\ol D$, $T=\opa S_V$ for some $(n,n-1)$-current $S_V$ with compact support in $V$, there exists an $(n,n-1)$-current $S$ with support in $\ol D$ satisfying $T=\opa S$.
\end{thm}
\begin{proof}
Assume $D$ has the $\ci$-Mergelyan property, then for any $g\in\oc(D)\cap\ci(\ol D)$, any $\varepsilon>0$ and any integer $N$, there exists a neighborhood $W$ of $\ol D$ and a function $f_N\in\oc(W)$ such that $\|g-f_N\|_{\cc^N(\ol D)}<\varepsilon$.

Let $T$ be an $(n,n)$-current with support in $\ol D$ such that, for any neighborhood $V$ of $\ol D$, there exists an $(n,n-1)$-current $S_V$ with compact support in $V$ satisfying $T=\opa S_V$. Since $T$ has compact support, it is of finite order $N$. Let $g\in\oc(D)\cap\ci(\ol D)$, the density hypothesis implies that for any $\varepsilon>0$, there exists $f\in\oc(V)$ for some neighborhood $V$ of $\ol D$ such that $\|g-f\|_{\cc^N(\ol D)}<\varepsilon$.

 Therefore
$$|<T,g>|\leq |<T,g-f>|+|<T,f>|\leq C\varepsilon+|<\opa S_V,f>|\leq C\varepsilon+|<S_V,\opa f>|\leq C\varepsilon.$$
So for any $g\in\oc(D)\cap\ci(\ol D)$, $<T,g>=0$ and since $H^{n,n}_{\ol D,cur}(X)$ is Hausdorff, we get $T=\opa S$ for some $(n,n-1)$-current $S$ with support in $\ol D$.
\end{proof}

\begin{cor}
Let $X$ be a complex manifold of complex dimension $n\geq 1$, $D\subset\subset X$ a relatively compact {\color{red} domain} with smooth boundary in $X$.  Assume $\ol D$ admits a neighborhood $V$ which is a $1$-convex extension of $\ol D$, then if  $T$ is an $(n,n)$-current with support in $\ol D$ such that, for any neighborhood $V$ of $\ol D$, $T=\opa S_V$ for some $(n,n-1)$-current $S_V$ with compact support in $V$, there exists an $(n,n-1)$-current $S$ with support in $\ol D$ satisfying $T=\opa S$.
\end{cor}

\subsection{The $L^2$ Runge density property}

\begin{defin}
A relatively compact domain $D$ with Lipschitz boundary in $X$ is \emph{$L^2$ $q$-Runge}, for $0\leq q\leq n-1$, if and only if, for any $0\leq p\leq n$, the space $Z^{p,q}_{L^2_{loc}}(X)$ of $L^2_{loc}$ $\opa$-closed $(p,q)$-forms in $X$ is dense in the space $Z^{p,q}_{L^2}(D)$ of  $L^2$ $\opa$-closed $(p,q)$-forms in $D$ for the $L^2$ topology on $D$.

For $q=0$, we will simply say that the domain  is \emph{$L^2$ Runge}, which means that the space $\oc(X)$ of holomorphic functions in $X$ is dense in the space $H^2(D)=\oc(D)\cap L^2(D)$ of $L^2$ holomorphic functions in $D$, for the $L^2$ topology on $D$.
\end{defin}

If $D\subset\subset X$ is a relatively compact domain in $X$, we denote by $H^{r,s}_{\ol D, L^2}(X)$ the  Dolbeault cohomology groups of $L^2$ forms with prescribed support in $\ol D$ and by ${H}^{r,s}_{\Phi, W^1}(X\setminus D)$ the Dolbeault cohomology groups of $W^1$ forms in $X\setminus D$ vanishing outside a compact subset of $X$.

\begin{thm}\label{L2runge}
Let $X$ be a non-compact complex manifold of complex dimension $n\geq 1$, $D\subset\subset X$ a relatively compact domain with Lipschitz boundary in $X$ and $q$ be a fixed integer such that $0\leq q\leq n-1$. Assume  that, for any $0\leq p\leq n$, $H^{n-p,n-q}_c(X)$ and  $H^{n-p,n-q}_{\ol D,L^2}(X)$ are Hausdorff. Then $D$ is an $L^2$ $q$-Runge domain in $X$ if and only if the natural map
$H^{n-p,n-q}_{\ol D,L^2}(X)\to H^{n-p,n-q}_c(X)$
is injective.
\end{thm}
\begin{proof}
Assume $D$ is $L^2$ Runge in $X$ and let $f\in L^2_{n-p,n-q}(X)$ with support contained in $\ol D$ be such that the cohomological class $[f]$ of $f$ vanishes in $H^{n-p,n-q}_c(X)$ , which means that there exists $g\in L^2_{n-p,n-q-1}(X)$ with compact support in $X$ such that $f=\opa g$. Since $H^{n-p,n-q}_{\ol D,L^2}(X)$ is Hausdorff, then $[f]=0$ in $H^{n-p,n-q}_{\ol D,L^2}(X)$ if and only if, for any form  $\varphi\in Z^{p,q}_{L^2}(D)$, we have $\int_D \varphi\wedge f=0$. But, as $D$ is $L^2$ $q$-Runge in $X$, there exists a sequence $(\varphi_k)_{k\in\nb}$ of $L^2_{loc}$ $\opa$-closed $(p,q)$-forms in $X$  which converges to $\varphi$ in $L^2(D)$. So
$$\int_D \varphi\wedge f=\lim_{k\to\infty}\int_X \varphi_k\wedge f=\lim_{k\to\infty}\int_X \varphi_k\wedge \opa g=\pm \lim_{k\to\infty}\int_X \opa\varphi_k\wedge g=0.$$

Conversely, by the Hahn-Banach theorem, it is sufficient to prove that, for any form $g\in Z^{p,q}_{L^2}(D)$ and any  $(n-p,n-q)$-form $\varphi$ in $L^2(D)$ such that $\int_D \varphi\wedge f=0$ for any form $f\in Z^{p,q}_{L^2_{loc}}(X)$, we have $\int_D \varphi\wedge g=0$. We still denote by $\varphi$ the extension of $\varphi$ as an $L^2$ form on $X$ with compact support in $\ol D$. Since $H^{n-p,n-q}_c(X)$ is Hausdorff, the hypothesis on $\varphi$ implies that there exists an $L^2$ $(n-p,n-q-1)$-form $\psi$ with compact support in $X$ such that $\varphi=\opa \psi$. The injectivity of the natural map
$H^{n-p,n-q}_{\ol D,L^2}(X)\to H^{n-p,n-q}_c(X)$ implies that there exists an $L^2$ $(n-p,n-q-1)$-form $\theta$ with compact support in $\ol D$ such that $\varphi=\opa_{\wt c}\theta$. Hence since the boundary of $D$ is Lipschitz, for any $g\in Z^{p,q}_{L^2}(D)$, we get
$$\int_D \varphi\wedge g=\int_D \opa_{\wt c}\theta\wedge g=\pm\int_D \theta\wedge \opa g=0.$$
\end{proof}

\begin{prop}\label{L2complement}
Let $X$ be a non-compact complex manifold of complex dimension $n\geq 2$, $D\subset\subset X$ a relatively compact domain in $X$ with Lipschitz boundary and $q$ be a fixed integer such that $0\leq q\leq n-2$. Assume  that, for some $0\leq p\leq n$,  $H^{n-p,n-q-1}_c(X)=0$. Then ${H}^{n-p,n-q-1}_{\Phi, W^1}(X\setminus D)=0$ if and only if the natural map
$H^{n-p,n-q}_{\ol D,L^2}(X)\to H^{n-p,n-q}_c(X)$
is injective.
\end{prop}
\begin{proof}
We first consider the necessary condition. Let $f\in L^2_{n-p,n-q}(X)$ be a $\opa$-closed form with support contained in $\ol D$ such that the cohomological class $[f]$ of $f$ vanishes in $H^{n-p,n-q}_c(X)$, by the Dolbeault isomorphism and the interior regularity of the $\opa$ operator, there exists $g\in W^1_{n-p,n-q-1}(X)$ and compactly supported such that $f=\opa g$. Since the support of $f$ is contained in $\ol D$, we have $\opa g=0$ on $X\setminus\ol D$. Therefore the vanishing of the group $H^{n-p,n-q-1}_{\Phi,W^1}(X\setminus D)$ implies that there exists $u\in W^1_{n-p,n-q-2}(X\setminus D)$ such that $\opa u=g$ on $X\setminus\ol D$. Since the boundary of $D$ is Lipschitz there exists $\wt u$ a $W^1$ extension of $u$ to $X$, we set $v=g-\opa\wt u$, then $v\in L^2_{n-p,n-q-1}(X)$ satisfies $f=\opa v$ and ${\rm supp}~v\subset\ol D$.

Conversely, let $g$ be a  $\opa$-closed $(n-p,n-q-1)$-form in $W^1_{n-p,n-q-1}(X\setminus D)$ which vanishes outside a compact subset of $X$ and $\wt g$ a $W^1$ extension of $g$ to $X$, then $\wt g$ has compact support in $X$ and $f=\opa\wt g$ is a form in $L^2_{n-p,n-q}(X)$ with support in the closure of $D$. By the injectivity of the natural map
$H^{n-p,n-q}_{\ol D,L^2}(X)\to H^{n-p,n-q}_c(X)$, there exists $u\in L^2_{n-p,n-q-1}(X)$ with support contained in $\ol D$ and such that $\opa u=f$. We set $v=\wt g-u$, $v$ is then an $L^2$ $\opa$-closed $(n-p,n-q-1)$-form with compact support in $X$ such that $v_{|_{X\setminus\ol D}}=g$. Since $H^{n-p,n-q-1}_c(X)=0$, by the Dolbeault isomorphism and the interior regularity of the $\opa$ operator, we have $v=\opa w$ with $w\in W^1_{n-p,n-q-2}(X)$ with compact support in $X$. Finally we get $g=v_{|_{X\setminus\ol D}}=\opa w_{|_{X\setminus\ol D}}$.
\end{proof}

\begin{cor}\label{L2rungecarac}
Let $X$ be a Stein  manifold of complex dimension $n\geq 2$ and $D\subset\subset X$ a relatively compact pseudoconvex domain with Lipschitz boundary in $X$. Then the following assertions are equivalent:

i) the domain $D$ is $L^2$ Runge,

ii) the natural map
$H^{n,n}_{\ol D,L^2}(X)\to H^{n,n}_c(X)$
is injective,

iii) ${H}^{n,n-1}_{\Phi,W^1}(X\setminus D)=0$.
\end{cor}
\begin{proof}
Since $X$ is Stein, we have $H^{n,n-1}_c(X)=0$ and $H^{n,n}_c(X)$ is Hausdorff. The domain $D$ being relatively compact, pseudoconvex in $X$,  we have $H^{0,1}_{L^2}( D)=0$ by the classical H\"ormander $L^2$ theory (see
\cite{Ho1, Ho}).
Since the boundary of $D$ is Lipschitz, the Serre duality implies that $H^{n,n}_{\ol D, L^2}(X)$ is Hausdorff.
The corollary follows then from Theorem \ref{L2runge} and Proposition \ref{L2complement}.
\end{proof}

Finally using the characterization of pseudoconvexity by means of $L^2$ cohomology and $L^2$ Serre duality, we can prove the following corollary.

\begin{cor}
Let $X$ be a Stein manifold of complex dimension $n\geq 2$  and $D\subset\subset X$ a relatively compact domain in $X$ with Lipschitz boundary such that $X\setminus D$ is connected. Then
the following assertions are equivalent:

(i) the domain $D$ is pseudoconvex and $L^2$-Runge in $X$;

(ii) $H^{n,r}_{\ol D, L^2}(X)=0$, for $2\leq r\leq n-1$, and the natural map
$H^{n,n}_{\ol D,L^2}(X)\to H^{n,n}_c(X)$ is injective;

(iii) ${H}^{n,q}_{\Phi,W^1}(X\setminus D)=0$, for all $1\leq q\leq n-1$.
\end{cor}
\begin{proof}
Using H\"ormander's $L^2$ theory for the necessary condition and Theorem 5.1 in \cite{FuLaSh} for the sufficient one, we get that a domain $D$, such that interior$(\ol D)=D$, is pseudoconvex if and only if $H^{0,q}_{L^2}(D)=0$ for all $1\leq q\leq n-1$. Applying $L^2$ Serre duality (see \cite{ChaShdual}), we get that if $D$ has a Lipschitz boundary, then $D$ is pseudoconvex if and only if $H^{n,q}_{\ol D,L^2}(X)=0$ for all $2\leq q\leq n-1$ and $H^{n,n}_{\ol D,L^2}(X)$ is Hausdorff.

To  get the equivalence between (i) and (ii), it remains to prove that the injectivity of the natural map
$H^{n,n}_{\ol D,L^2}(X)\to H^{n,n}_c(X)$  implies that $H^{n,n}_{\ol D,L^2}(X)$ is Hausdorff and to apply Theorem \ref{L2runge}.

Let $f$ be an $L^2_{loc}$ $(n,n)$-form with support in $\ol D$ such that $\int_D f\varphi=0$ for any $L^2$ holomorphic function $\varphi$ on $D$. In particular $\int_X f\varphi=0$ for any $L^2_{loc}$ holomorphic function $\varphi$ on $X$ and $X$ being Stein, $H^{n,n}_c(X)$ is Hausdorff and therefore $f=\opa u$ for some $L^2_{loc}$ $(n,n-1)$-form $u$ with compact support in $X$, i.e. $[f]=0$ in $H^{n,n}_c(X)$. By the injectivity of the map $H^{n,n}_{\ol D,L^2}(X)\to H^{n,n}_c(X)$, we get that $f=\opa g$ for some $L^2_{loc}$ $(n,n-1)$-form $g$ with support in $\ol D$, which ends the proof.

Let us prove now the equivalence between (ii) and (iii).
It follows from Theorem 4.7 in \cite{FuLaSh} that, for all $2\leq q\leq n-1$, $H^{n,q}_{\ol D,L^2}(X)=0$ if and only if ${H}^{n,q{\color{red} -1}}_{W^1_{loc}}(X\setminus\ol D)=0$.
It remain to prove that, for all
$1\leq q\leq n-2$, $H^{n,q}_{\Phi,W^1}(X\setminus D)=0$ if and only if $H^{n,q}_{W^1_{loc}}(X\setminus D)=0$ and that $H^{n,n-1}_{\Phi,W^1}(X\setminus D)=0$, implies $H^{n,n-1}_{W^1_{loc}}(X\setminus D)$ is Hausdorff. Using the Dolbeault isomorphism, this can be done in the same way as for Lemma \ref{complementsupport}.
Then we apply Proposition \ref{L2complement} to get the result.
\end{proof}

\begin{defin}
A relatively compact domain $D$ in $X$ has the \emph{$L^2$-Mergelyan property} if and only if the space $\oc(\ol D)$ of germs of holomorphic functions on $\ol D$ is dense in the space $\oc(D)\cap L^2(D)$ of holomorphic functions on $D$ for the $L^2$ topology in $D$.
\end{defin}

Let us consider now the case when the closure of $D$ is a Stein compactum.
We prove the following result on the $L^2$-Mergelyan property, which is slightly stronger than Theorem 26 in the survey paper \cite{FoForWo}.

 \begin{thm}\label{th:L2Mergelyan}  Let $X$ be a complex   manifold of complex dimension $n$ and let $ D\subset\subset X$  a
relatively compact pseudoconvex domain with  Lipschitz  boundary in $X$ whose closure $\overline D$ has a
Stein neighborhood basis. Then $\mathcal O(\overline D)$ is dense in $H^2(D)=L^2(D)\cap \mathcal O(D)$.
 \end{thm}

 \begin{proof} Let $(D_\nu)$ be the Stein neighborhood basis of $\overline D$. Let  $h\in H^2(D)$. Since $bD$ is Lipschitz, using Friedrichs' Lemma (see (i) in Lemma 4.3.2 in \cite{ChSh}),
 there exists a sequence of functions $h_\nu\in L^2(D_\nu)$ such that
  $h_\nu\to h$ in $L^2(D)$
 and
 $$\|\bar\partial h_\nu\|_{L^2(D_\nu)}\to 0.$$
 Each $D_\nu$ is pseudoconvex. From the H\"ormander's $L^2$ existence theorem, there exists $u_\nu\in L^2(D_\nu)$ such that $\bar\partial u_\nu=\bar\partial h_\nu$
 on $D_\nu$ with
 $$\|u_\nu\|_{L^2(D_\nu)}\le C\|\bar\partial h_\nu\|_{L^2(D_\nu)}\to 0$$
 where $C$ is independent of $\nu$.

 Let $H_\nu=h_\nu-u_\nu$. Then $H_\nu\to h$ in $L^2(D)$ and $H_\nu\in H^2(D_\nu)\subset \mathcal O(D_\nu)$.
 The theorem is proved.
\end{proof}

From the Oka-Weil theorem associated to Theorem \ref{th:L2Mergelyan}, it is easy to derive the following sufficient geometric condition for a pseudoconvex domain to be $L^2$-Runge in $X$.

\begin{thm}\label{geomL2}
Let $X$ be a Stein manifold of complex dimension $n\geq 2$ and let $ D\subset\subset X$ be a
relatively compact pseudoconvex domain with  Lipschitz  boundary in $X$. Assume the closure $\overline D$ of $D$ has a $\oc(X)$-convex neighborhood basis, then $D$ is $L^2$-Runge in $X$.
\end{thm}
\begin{proof}
Since the closure $\overline D$ of $D$ has a $\oc(X)$-convex neighborhood basis, it admits in particular a Stein neighborhood basis and we can apply Theorem \ref{th:L2Mergelyan}. therefore if $f\in\oc(D)\cap L^2(D)$ and $\varepsilon>0$ is a real number, there exists a neighborhood $V$ of $\ol D$ and a holomorphic function $g\in\oc(V)$ such that $\|f-g\|_{L^2(D)}\leq\frac{\varepsilon}{2}$. By hypothesis there exists $U\subset\subset V$ a neighborhood of $\ol D$ such that $\ol U$ is $\oc(X)$-convex. So  we can apply the Oka-Weil theorem (see Corollary \ref{okaweil}) and then there exists a function $h\in\oc(X)$ such that
$$\|g-h\|_{L^2(U)}\leq C \|g-h\|_\infty\leq\frac{\varepsilon}{2}.$$ Finally  $$\|f-h\|_{L^2(D)}\leq\|f-g\|_{L^2(D)}+\|g-h\|_{L^2(U)}\leq\varepsilon.$$
\end{proof}

The next result on the $\opa$ problem with mixed bondary conditions in an annulus is a direct consequence of Theorem \ref{geomL2} and the characterization of $L^2$-Runge domains in a Stein manifold of Corollary \ref{L2rungecarac}.
\begin{cor}
Let $X$ be a Stein manifold of complex dimension $n\geq 2$ and let $ D\subset\subset X$ be a
relatively compact pseudoconvex domain with  Lipschitz  boundary in $X$. Assume the closure $\overline D$ of $D$ has a $\oc(X)$-convex neighborhood basis, then
${H}^{n,n-1}_{\Phi,W^1}(X\setminus D)=0$.
\end{cor}

As for the $\ci$-Mergelyan property, we can relate the $L^2$-Mergelyan property with the solvability of the $\opa$-equation with prescribed support.

 \begin{thm}\label{mergelyanpropL2}
Let $X$ be a complex manifold of complex dimension $n\geq 1$, $D\subset\subset X$ a relatively compact domain with Lipschitz boundary in $X$. Assume  $\ol D$ admits a neighborhood basis of $1$-convex open subsets. Assume that for all $(n,n)$-form $f$ in $L^2(X)$ with support in $\ol D$ such that, for any sufficiently small neighborhood $V$ of $\ol D$, $f=\opa g_V$ for some $(n,n-1)$-form $g_V$ in $L^2(X)$ with compact support in $V$, there exists an $(n,n-1)$-form $g$ in $L^2(X)$ with support in $\ol D$ satisfying $f=\opa g$, then $D$ has the $L^2$-Mergelyan property.
\end{thm}
\begin{proof}
Assume the cohomological condition holds, we apply the Hahn-Banach theorem. Let $f$ be an $(n,n)$-form $f$ in $L^2(X)$ with support in $\ol D$ such that {\color{red} $\int_D f\varphi=0$} for any $\varphi\in {\color{red} \oc(\ol D)}$. For any $1$-convex neighborhood $V$ of $\ol D$, $H^{n,n}_c(V)$ is Hausdorff. Hence there exists an $(n,n-1)$-form $g_V$ in $L^2(X)$ with compact support in $V$ such that $f=\opa g_V$. Using the hypothesis, we get that $f=\opa g$ for some $(n,n-1)$-form $g$ in $L^2(X)$ with support in $\ol D$. Let $\psi\in H^2(D)$, then
{\color{red}
$$\int_D f\psi=\int_D (\opa g)\psi=\pm \int_D g\wedge\opa\psi=0.$$
}
\end{proof}

Conversely we get the next theorem.

\begin{thm}\label{mergelyancompactL2}
Let $X$ be a complex manifold of complex dimension $n\geq 1$, $D\subset\subset X$ a relatively compact pseudoconvex domain with Lipschitz boundary in $X$. Assume $D$ has the $L^2$-Mergelyan property, then if  $f$ is an $(n,n)$-form in $L^2(X)$ with support in $\ol D$ such that, for any neighborhood $V$ of $\ol D$, $f=\opa g_V$ for some $(n,n-1)$-form $g_V$ in $L^2(X)$ with compact support in $V$, there exists an $(n,n-1)$-form $g$ in $L^2(X)$ with support in $\ol D$ satisfying $f=\opa g$.
\end{thm}
\begin{proof}
Assume $D$ has the $L^2$-Mergelyan property, then for any $\varphi\in H^2(D)$ and any $\varepsilon>0$, there exists a neighborhood $W$ of $\ol D$ and a function $\psi\in\oc(W)$ such that $\|\varphi-\psi\|_{L^2(D)}<\varepsilon$.

Since $D$ is pseudoconvex with Lipschitz boundary, $H^{n,n}_{\ol D,L^2}(X)$ is Hausdorff.

Let $f$ be an $(n,n)$-form in $L^2(X)$ with support in $\ol D$ such that, for any neighborhood $V$ of $\ol D$, there exists an $(n,n-1)$-form $g_V$ in $L^2(X)$ with compact support in $V$ satisfying $f=\opa g_V$. Let $\varphi\in H^2(D)$, the density hypothesis implies that for any $\varepsilon>0$, there exists $\psi\in\oc(V)$ for some neighborhood $V$ of $\ol D$ such that $\|\varphi-\psi\|_{L^2(D)}<\varepsilon$.

 Therefore
{\color{red}
$$|\int_D f\varphi|\leq |\int_D f(\varphi-\psi)|+|\int_D f\psi|\leq C\varepsilon+|\int_D (\opa g_V)\psi>|\leq C\varepsilon+|\int_D g_V\wedge\opa\psi|\leq C\varepsilon.$$
}
So for any $\varphi\in H^2(D)$, $<f,\varphi>=0$ and since $H^{n,n}_{\ol D,L^2}(X)$ is Hausdorff, we get $f=\opa g$ for some $(n,n-1)$-form $g$ in $L^2(X)$ with support in $\ol D$.
\end{proof}

As a direct consequence of Theorems \ref{mergelyancompactL2} and \ref{th:L2Mergelyan}, we obtain the following result.

\begin{cor}
Let $X$ be a complex manifold of complex dimension $n\geq 1$, $D\subset\subset X$ a relatively compact pseudoconvex domain with Lipschitz boundary in $X$, whose closure $\ol D$ has a Stein neighborhood basis. Then if  $f$ is an $(n,n)$-form in $L^2(X)$ with support in $\ol D$ such that, for any neighborhood $V$ of $\ol D$, $f=\opa g_V$ for some $(n,n-1)$-form $g_V$ in $L^2(X)$ with compact support in $V$, there exists an $(n,n-1)$-form $g$ in $L^2(X)$ with support in $\ol D$ satisfying $f=\opa g$.
\end{cor}

\section{The $1$-dimensional case}\label{sdim1}

In this section we consider more precisely the case when $X$ is a Riemann surface. In particular we will relate the classical $1$-dimensional Runge's theorem with some properties of some Dolbeault cohomology groups and see that in Riemann surfaces the different notions Runge, $\ci$-Mergelyan and $L^2$-Mergelyan in $X$ are all equivalent for a domain $D$ with sufficiently smooth boundary.

  The classical Oka-Weil theorem (see Corollary \ref{okaweil}), asserts that, in a Stein manifold $X$ of complex dimension $n\geq 2$, a sufficient condition for a compact subset $K$ to be Runge in $X$ is to be $\oc(X)$-convex. It follows from the maximum principle that, if a compact subset $K$ is $\oc(X)$-convex, then $X\setminus K$ has no relatively compact connected components. This topological property is the key point of the study of holomorphic approximation in Riemann surfaces.

The results of the next sections are reformulations in terms of Dolbeault cohomology groups of the classical results in Riemann surfaces. The ideas of the proofs have already appear in \cite{Marunge} and in the proof of Theorem 1.3.1 in \cite{Ho}.

\subsection{Runge's theorem}

Let $X$ be a connected Riemann surface and $K\subset\subset X$ be a compact subset of $X$. We denote by $H^{1,1}_{K,cur}(X)$ the Dolbeault cohomology group with prescribed support in $K$ for  $(1,1)$-currents and by $\check{H}^{1,0}_\Phi(X\setminus K)$ the set of holomorphic $(1,0)$-forms $h$ on $X\setminus K$, vanishing outside some compact subset of $X$ and such that $h=S_{|{X\setminus K}}$ for some $(1,0)$-current $S$ on $X$. For references in this case, we refer to the survey paper \cite{FoForWo}.

\begin{thm}\label{rungecompact1}
Let $X$  be a connected Riemann surface and $K\subset\subset X$ be a compact subset of $X$.  Then the following assertions are equivalent:

(i) $\check{H}^{1,0}_\Phi(X\setminus K)=0$;

(ii) the natural map $H^{1,1}_{K,cur}(X)\to H^{1,1}_c(X)$ is injective;

(iii) $K$ is Runge in $X$;

(iv) $X\setminus K$ has no relatively compact connected components.
\end{thm}
\begin{proof}
We first prove that (i) implies (ii). Assume $\check{H}^{1,0}_\Phi(X\setminus K)=0$ and let $T$ be a $\opa$-closed $(1,1)$-current on $X$ with support contained in $K$ such that $T=\opa S$ for some $(1,0)$-current with compact support in $X$, then $h=S_{|{X\setminus K}}$ belongs to $\check{H}^{1,0}_\Phi(X\setminus K)$ and hence $h=0$ on $X\setminus K$, which means that the support of $S$ is contained in $K$.

Note that if the natural map $H^{1,1}_{K,cur}(X)\to H^{1,1}_c(X)$ is injective, then  for any $V$ belonging to a neighborhood basis of $K$ the natural map $H^{1,1}_c(V)\cap (\ec'_K)^{1,1}(X)\to H^{1,1}_c(X)$ is injective. Then it remains to apply Theorem \ref{rungeK} the get that (ii) implies (iii).

It follows from the maximum principle that (iii) implies (iv).

Assume (iv) is satisfied, therefore if $h\in \check{H}^{1,0}_\Phi(X\setminus K)$, then $h$ vanishes on an open subset of each connected component of $X\setminus K$ and by analytic continuation $h=0$ on $X\setminus K$, which implies (i).
\end{proof}

Note that the equivalence between (iii) and (iv) in Theorem \ref{rungecompact1} is exactly the classical Runge's theorem.

\subsection{Mergelyan properties}

As mentioned in  previous sections, holomorphic approximation is directly related to the solvability of the $\opa$-equation with compact or prescribed support in top degree. But if the manifold $X$ is $1$-dimensional, this means bidegree $(1,1)$, which is very special.

Following section 1 in \cite{LaShHartogs}, we have:

\begin{prop}\label{support}
Let $X$ be a complex manifold and $T$ be a $\opa$-exact $(1,1)$-current on $X$. If $\Omega^c$ denotes a connected component of $X\setminus{\rm supp}~T$ and if $S$ is a $(1,0)$-current on $X$ such that $\opa S=T$, then either ${\rm supp}~S\cap\Omega^c=\emptyset$ or $\Omega^c\subset{\rm supp}~S$.
\end{prop}

Assume the complex manifold $X$ is non-compact, then from Proposition \ref{support}, we get that, if $T$ is a $(1,1)$-current with compact support in $X$ such that the cohomology class $[T]$ of $T$ in $H^{1,1}_c(X)$ vanishes, the support of the unique solution $S$ with compact support in $X$ of the equation $\opa S=T$ is contained in the union of the support of $T$ and the relatively compact connected components of $X\setminus {\rm supp}~T$. Moreover using the regularity of the $\opa$ operator, we get

\begin{prop}\label{trou}
Let $X$ be a connected, non-compact, complex manifold and $D$ a relatively compact domain in $X$ such that  $X\setminus D$ has no relatively compact connected component in $X$. Then injectivity holds for all the natural maps

1) $H^{1,1}_{\ol D,cur}(X)\to H^{1,1}_c(X)$,

2) $H^{1,1}_{\ol D,L^2}(X)\to H^{1,1}_c(X)$.
\end{prop}

Using Theorems \ref{stronglyrunge} and \ref{L2runge}, we then obtain:

\begin{cor}
Let $X$ be a connected non-compact complex manifold of complex dimension $1$ and $D\subset\subset X$ a relatively compact domain with Lipschitz boundary in $X$.
 Assume $X\setminus D$ has no relatively compact connected component in $X$, then $D$ is $\ci$ and $L^2$-Runge in $X$.
\end{cor}


\begin{thebibliography}{10}

\bibitem{ChaShdual}
D.~Chakrabarti and M.-C. Shaw, \emph{${L}^2$ {S}erre duality on domains in
  complex manifolds and applications}, Trans. A.M.S. \textbf{364} (2012),
  3529--3554.

\bibitem{ChSh}
S.-C. Chen and M.-C. Shaw, \emph{{P}artial differential equations in several
  complex variables}, Studies in Advanced Math., vol.~19, AMS-International
  Press, 2001.

\bibitem{ChSt}
E. M. Chirka and E. L. Stout, \emph{Removable singularities in the boundary},
  Contributions to Complex Analysis and Analytic Geometry, Aspects of
  Mathematics, vol. E26, 1994, pp.~43--104.

\bibitem{FoForWo}
J.~E. Fornaess, F.~Forstneric, and E.~F. Wold, \emph{Holomorphic approximation:
  the legacy of {W}eierstrass, {R}unge, {O}ka-{W}eil, and {M}ergelyan},
  Preprint.

\bibitem{FuLaSh}
S.~Fu, C.~Laurent-Thi{\'e}baut, and M.C. Shaw, \emph{Hearing pseudoconvexity in
  lipschitz domains with holes via $\opa$}, Math. Zeit. \textbf{287} (2017),
  1157--1181.

\bibitem{HeLe2}
G.~M. Henkin and J.~Leiterer, \emph{{A}ndreotti-{G}rauert theory by integral
  formulas}, Progress in Math., vol.~74, Birkha{\"u}ser, Basel, Boston, Berlin,
  1988.
 \bibitem  {Ho1} L. H\"ormander, {\it
 $L^2$ Estimates and Existence Theorems for the $\overline\partial$ Operator}, {\it Acta Math.} {\textbf 113} (1965), 89-152.

\bibitem{Ho}
L.~H{\"o}rmander, \emph{An introduction to complex analysis in several complex
  variables}, Van Nostrand, Princeton, N.J., 1990.

\bibitem{Lalivre}
C.~Laurent-Thi{\'e}baut, \emph{Th{\'e}orie des fonctions holomorphes de
  plusieurs variables}, Savoirs actuels, {I}nter{E}ditions/{CNRS} {E}ditions,
  Paris, 1997.

\bibitem{Lacalotte}
\bysame, \emph{Sur l'{\'e}quation de {C}auchy-{R}iemann tangentielle dans une
  calotte strictement pseudoconvexe}, Int. Journal of Math. \textbf{16} (2005),
  1063--1079.

\bibitem{LaLeMathAnn}
C.~Laurent-Thi{\'e}baut and J.~Leiterer, \emph{On the {H}artogs-{B}ochner
  extension phenomenon for differential forms}, Math. Ann. \textbf{284} (1989),
  103--119.

\bibitem{LaShdualiteL2}
C.~Laurent-Thi{\'e}baut and M.C. Shaw, \emph{On the {H}ausdorff property of
  some {D}olbeault cohomology groups}, Math. Zeitschrift \textbf{274} (2013),
  1165--1176.

\bibitem{LaShHartogs}
\bysame, \emph{Solving $\opa$ with prescribed support on hartogs triangles in
  $\cb^2$ and $\cb\pb^2$}, Trans. A.M.S. \textbf{371} (2019), 6531--6546.

\bibitem{Lu}
G.~Lupacciolu, \emph{Characterization of removable sets in strongly
  pseudoconvex boundaries}, Ark. Mat. \textbf{32} (1994), 455--473.

\bibitem{Marunge}
B.~Malgrange, \emph{Existence et approximation des solutions des \'equations
  aux d\'eriv\'ees partielles et des \'equations de convolution}, Ann. Inst.
  Fourier \textbf{6} (1956), 271--355.

\bibitem{Saext}
S.~Sambou, \emph{R{\'e}solution du $\opa_b$ pour les courants prolongeables
  d{\'e}finis dans un anneau}, Ann. Fac. Sc. Toulouse \textbf{11} (2002),
  105--129.

\end{thebibliography}
\providecommand{\bysame}{\leavevmode\hbox to3em{\hrulefill}\thinspace}
\providecommand{\MR}{\relax\ifhmode\unskip\space\fi MR }
\providecommand{\MRhref}[2]{%
  \href{http://www.ams.org/mathscinet-getitem?mr=#1}{#2}
}
\providecommand{\href}[2]{#2}

\enddocument

\end